\documentclass[11pt]{article}

\usepackage{enumerate}
\usepackage{latexsym} 
\usepackage{theorem}
\usepackage{graphics}
\usepackage{graphicx}
\usepackage{amsmath}
\usepackage{amssymb}  
\usepackage[english]{babel} 
\usepackage{comment}
\usepackage{color}
\usepackage{tikz}
\usetikzlibrary{arrows}
\usepackage{amsfonts}
\usepackage{setspace}
\usepackage{fullpage}
\usepackage{enumitem}
\usepackage{hyperref}
\usepackage{microtype}
\usepackage{enumerate}
\usepackage{authblk}

\newtheorem{theorem}{Theorem}[section]
\newtheorem{lemma}[theorem]{Lemma}
\newtheorem{problem}[theorem]{Problem}

\newtheorem{proposition}[theorem]{Proposition}

\newtheorem{conjecture}[theorem]{Conjecture}

\newtheorem{definition}[theorem]{Definition}

\tikzset{>=latex}
\tikzstyle{vertex}=[circle, draw, inner sep=1pt, minimum size=12pt]
\newcommand{\vertex}{\node[vertex]}

\tikzstyle{minivertex}=[circle, draw, inner sep=0pt, minimum size=3pt]

\usetikzlibrary{calc,decorations.pathmorphing,decorations.text,fixedpointarithmetic}

\newcounter{claim}

{\refstepcounter{claim}\vspace{1ex}\noindent{(\it\arabic{claim}){#1}{}}\it}{\vspace{1ex}}

\newenvironment{proof}[1][]%
 {\noindent {\setcounter{claim}{0}\sc proof ---
   }{#1}{}}{\hfill$\Box$\vspace{2ex}}

	{\noindent {}{#1}{}}{ This proves Claim~\arabic{claim}.\vspace{1ex}}

\tikzset{every node/.style={circle,draw=black!80, line width=0.6mm, inner sep=0mm, minimum size=2.5mm}}

\newcommand{\F}{Forb_{ind}}
\newcommand{\ova}{\overrightarrow}
\newcommand{\ovlra}{\overleftrightarrow}
\newcommand{\ra}{\rightarrow}

\newcommand{\dic}{\vec \chi}
\newcommand{\sm}{\setminus}
\newcommand{\mc}{\mathcal}
\newcommand{\overbar}[1]{\mkern 1.7mu\overline{\mkern-1.7mu#1\mkern-1.7mu}\mkern 1.7mu}
\newcommand{\C}[1]{\ova{C_{#1}}}
\newcommand{\Pb}[1]{{\bf P#1}}

\def\Gyar{Gy\'arf\'as~}

\begin{document}
\title{Extension of Gy\'arf\'as-Sumner conjecture to digraphs}

\author[1]{Pierre Aboulker}
\author[2]{Pierre Charbit} 
\author[2]{Reza Naserasr}
	
\affil[1]{DIENS, \'Ecole normale sup\'erieure, CNRS, PSL University, Paris, France.}
\affil[2]{Université de Paris, CNRS, IRIF, F-75006, Paris, France. E-mail addresses: \{reza,charbit\}@irif.fr}

\maketitle

\begin{abstract}

The dichromatic number of a digraph $D$ is the minimum number of colors needed to color its vertices  in such a way that each color class induces an acyclic digraph. As it generalizes the notion of the chromatic number of graphs, it has been a recent center of study.
In this work we look at possible extensions of Gy\'arf\'as-Sumner conjecture. 
More precisely, we propose as a conjecture a simple characterization of finite sets $\mathcal F$ of digraphs such that every oriented graph with sufficiently large dichromatic number must contain a member of $\mathcal F$ as an induce subdigraph. 

Among notable results, we prove that oriented triangle-free graphs without a directed path of length $3$ are $2$-colorable. If condition of ``triangle-free'' is replaced with ``$K_4$-free'', then we have an upper bound of $414$. We also show that an orientation of complete multipartite graph with no directed triangle is 2-colorable. To prove these results we introduce the notion of \emph{nice sets} that might be of independent interest.

{\bf Keywords: dichromatic number, hereditary classes of digraphs} 

\end{abstract}


\section{Introduction}

Despite the fact that the chromatic number of graphs is arguably the most studied invariant in graph theory, there are still many questions about chromatic number for which we do not have today a satisfying answer. In particular, a lot of work has been done about the following interrogation : what induced substructures are expected to be found inside a graph if we assume it has very large chromatic number? Or equivalently what are the minimal families $\mc F$ such that the class of graphs that do not contain any graph in $\mc F$ as an induced subgraph has bounded chromatic number? Since cliques have unbounded chromatic number and do not contain any induced subgraph other than cliques themselves, it is clear that such an $\mc F$ must contain a clique. On the other hand, Erd\H os's celebrated result on the existence of graphs of high girth and high chromatic number \cite{Erd59} implies that if $\mc F$ is finite, then at least one member of $\mc F$ must be a forest. These two facts constitute the ``only if" part of the following tantalizing and still widely open conjecture of Gy\'arf\'as and Sumner (see~\cite{SS20} for a survey on known results). 

\begin{conjecture}[Gy\'arf\'as-Sumner, \cite{Gya87, Sum81}]
 Given two graphs $F_1$ and $F_2$ the class of graphs with no induced $F_1$ or $F_2$ has bounded chromatic number if and only if one of $F_1,F_2$ is a complete graph and the other is a forest.
\end{conjecture}

The goal of this paper is to study analogous questions in the setting of directed graphs.\\

 A directed graph, or in short \emph{digraph} $D=(V, A)$, is defined similarly to graph : $V$ is the set of vertices, but here $A$ is a set of {\em ordered pairs} from $V$ which are called \emph{arcs}. Thus in this work \textit{loops}, \textit{multiedges} and \textit{multiarcs} are not considered, but if $x$ and $y$ are two vertices of a digraph, both arcs $xy$ and $yx$ might exist.
 A digraph in which there is at most one arc between each pair of vertices is called an \emph{oriented graph}. A graph $G$ can be viewed as a \textit{symmetric} digraph, that is a digraph in which if $xy$ is an arc, then $yx$ is too. This digraph will be denoted by $\ovlra{G}$. Given a digraph $D=(V,A)$ the \emph{underlying} graph of $D$ is a graph on $V$ where $xy$ is an edge if either $xy$ or $yx$ is an arc of $D$. If $xy$ is an arc we say that $x$ \textit{sees} $y$. An oriented cycle in which indegree and outdegree of each vertex is 1, is a \emph{directed cycle}. 
Directed cycle of length $k$ is denoted by $\ova{C_{k}}$ and the directed cycle on $2$ vertices in sometime called a \textit{digon}. 

Observing that each edge of the unoriented graph $G$ is replaced by an oriented 2-cycle in $\ovlra{G}$, the following is a natural generalization of proper coloring of graphs introduced in 1982 by Neumann-Lara~\cite{NL82} : an \emph{acyclic coloring} of a digraph $D$ is an assignment of colors to the vertices of $D$ such that no color class contains a directed cycle. It is an easy exercise to check that proper colorings of $G$ are the same as acyclic colorings of $\ovlra{G}$.
Extending the notion of chromatic number, the \textit{acyclic chromatic number}, or simply \emph{dichromatic number}, of a digraph $D$, denoted $\dic(D)$, is defined to be the smallest number of colors required for an acyclic coloring of $D$.  
Being viewed as a natural generalization of the chromatic number to digraphs, the dichromatic number has been the center of attention recently. See  \cite{ACL19, BHL18,  HK15, H17,  KS20, LM17} for examples of recent works on this subject.

We note that some authors refer to the dichromatic number of a digraph as simply the chromatic number of a digraph. However, as we will consider the chromatic number of the underlying graph, we reserve the notation $\chi(D)$ for the chromatic number of the underlying graph of $D$. Moreover, an acyclic coloring of $D$ will be called a \emph{dicoloring} of $D$.
Observe  that for every digraph $D$, $\dic(D) \leq \chi(D)$. 
\medskip

Given a set $\mc F$ of digraphs we denote by $\F(\mc F)$  the set of digraphs which have no member of $\mc F$ as an induced subdigraph. If $\mc F$ is explicitly given by a list of digraphs $F_{1},F_{2},\ldots, F_{k}$ we will write $\F(F_{1},F_{2},\ldots, F_{k})$ instead of the heavier $\F(\{F_{1},F_{2},\ldots, F_{k}\})$. For example, if we denote by $\ova{K_2}$ the oriented graph on two vertices with one arc, $\F(\ova{K_2})$ is the class of symmetric digraphs which is an equivalent representation of the class of graphs. 
A class $\mc C$ of digraphs is said to be \emph{hereditary} if for every digraph $D\in \mc C$, every induced subdigraph $D'$ of $D$ is in $\mc C$. It is clear from the definition that $\F(\mc F)$ is a hereditary class of digraph for any choice of $\mc F$. Moreover, every hereditary class $\mc C$ of digraphs can be presented as $\F(\mc F)$ where $\mc F$ is the set of minimal digraphs not belonging to $\mc C$.  This set $\mc F$ is not necessarily finite, but in this work we are only interested in the cases where $\mc F$ is finite.

Given a class of digraphs  $\mathcal{C}$, we define the \textit{dichromatic number of $\mathcal C$} as $\dic(\mathcal{C})=\max\{\dic(D) \mid D\in \mathcal{C}\}$ with understanding that $\dic(\mathcal{C})=\infty$ when it is not bounded.  

\medskip

In this paper, we try to extend \Gyar and Sumner interrogations to the world of digraphs, so we investigate the following problem:

\begin{problem}\label{prob:TheMainProblem}
	What are the finite sets $\mathcal{F}$ of digraphs for which the class $\F(\mathcal{F})$ 
	has bounded dichromatic number?
\end{problem}

In the next section we will explain why the results of~\cite{hero}, that constitute the starting point of our work, can be seen as a partial answer of Problem~\ref{prob:TheMainProblem}.

\section{Tournaments, Heroes and Heroic Sets}

A  \emph{tournament} is an oriented complete graph. Whereas complete graphs are somehow trivial objects regarding chromatic number, tournaments are already a complex and rich family with regards to dichromatic number. Observe for example that the {\em transitive tournament} on $n$ vertices (i.e. the unique up to isomorphism tournament on $n$ vertices that contains no directed cycle), denoted by $TT_{n}$, has, by the definition, dichromatic number $1$.
On the other hand there exists constructions of tournaments of arbitrarily large dichromatic number, as we explain now.

Given two digraphs $H_1$ and $H_2$ on disjoint sets of vertices, we denote by $H_1\Rightarrow H_2$ the digraph obtained from disjoint union of $H_1$ and $H_2$ by adding an arc from each vertex of $H_1$ to each vertex of $H_2$. Given $k$ directed graphs $D_1, D_2, \ldots, D_{k}$ we denote by $C_k(D_1, D_2, \ldots, D_{k})$ the digraph built as follows: for $i =1, \dots, k-1$  $D_i \Rightarrow D_{i+1}$, and  $D_{k} \Rightarrow D_1$. When all $D_i$'s are isomorphic to a digraph $D$ we may simply write $C_{k}(D)$ in place of $C_k(D_1, D_2, \ldots, D_{k})$. 

The following folklore theorem is straightforward to prove.

\begin{theorem}\label{coro:C_k(D)}
	Given a digraph $D$ and an integer $k\geq 3$, $\dic(C_k(D))=\dic(D)+1$.
\end{theorem}

Using this, we can give the aforementioned construction of tournaments of arbitrarily large dichromatic number : consider $(D_{k})_{k\in \mathbb N}$ given by the induction $D_{1}=K_{1}$, $D_{k}=C_{3}(D_{k-1},D_{k-1},K_{1})$. It satisfies $\dic(D_{k})=k$ (see Figure \ref{fig:X3}).

\begin{figure}[htbp]
\begin{center}
\includegraphics[width=0.6\textwidth,page=1]{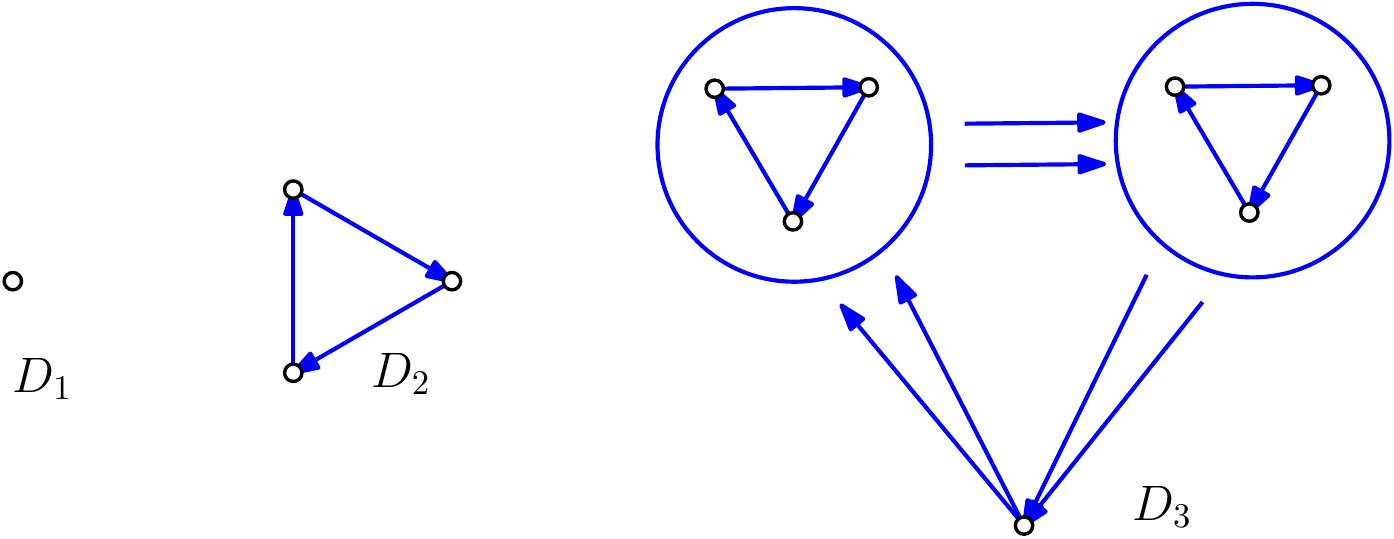}
\caption{Digraphs $D_{1}$,$D_{2}$,$D_{3}$}
\label{fig:X3}
\end{center}
\end{figure}

 Therefore Problem \ref{prob:TheMainProblem} is already meaningful when restricted to the family of tournaments and in a seminal paper~\cite{hero},  Berger et al. give a full characterization of tournaments $H$ such that the class of tournaments not containing $H$ has bounded dichromatic number. They call these tournaments {\em heroes}. 

\begin{theorem}\cite{hero}\label{thm:hero}
A tournament is a hero if and only if it can be constructed by the following inductive rules:
\begin{itemize}
	\item $K_1$ is a hero.
	\item If $H_1$ and $H_2$ are heroes, then $H_1 \Rightarrow H_2$ is also a hero.
	\item If $H$ is a hero, then for every $k \ge 1$, the tournaments $C_3(H, TT_k, K_1)$ and $C_3(TT_k, H, K_1)$ are both heroes. 
\end{itemize}

\end{theorem}

Extending the notion of hero, we say that a set $\mathcal{F}$ of digraphsis \emph{heroic} if $\F(\mathcal{F})$ has bounded dichromatic number. Problem \ref{prob:TheMainProblem} is then about characterizing finite heroic sets. 

We denote by $\overbar K_k $ the digraph made of $k$ vertices and no arc. Observing that $\F\{\ovlra K_2 , \overbar K_2\}$ is the set of all tournaments, a hero $H$ is then a digraph (necessarily a tournament) such that  $\{\ovlra{K_2}, \overbar K_2, H\}$ is heroic.

Hence, Theorem~\ref{thm:hero} characterizes all heroic sets of size three containing  $\ovlra K_2$ and $\overbar K_2$. However, it should be noted that minimal heroic sets of size  at least four containing $\ovlra K_2$ and $\overbar K_2$ exist. 
An example of such a set  is implicitly provided in \cite{hero}. Let $H_1$, $H_2$ and $H_3$ be the tournaments depicted in Figure~\ref{fig:H123}. Furthermore, let $H_4$ be $C_3(\ova{K_2},\ova{K_2},\ova{K_2})$ and let $H_5$ be $C_3(\ova{C_3},\ova{C_3},K_{1})$.
In~\cite{hero}, it is proved that the set $\{H_1,H_2, H_3,H_4,H_5\}$ is  the set of minimal non-hero tournaments, i.e., any tournament which is not a hero contains one of these five tournaments as a subdigraph. Using our terminology, it means that the set of heroes is precisely $\F\{\ovlra K_2, \overbar K_2,H_1,H_2, H_3,H_4,H_5\}$. 
Moreover it is easily observed that every hero is 2-colorable, so $\{\overbar{K_2}, \ovlra{K_2},H_1,H_2, H_3,H_4,H_5\}$ is a heroic set. We don't know if it is minimal or not but, since none of the $H_i$ are heroes  and by Theorem~\ref{thm:hero}, a minimal heroic subset of it has at least four elements. An interesting direction of research then would be to characterize all finite and minimal families $\mathcal{H}=\{ H_1, H_2, \ldots, H_k\}$ of tournaments such that $\{\overbar{K_2}, \ovlra{K_2}, H_1, H_2, \ldots, H_k\}$ is a heroic set. The main result of \cite{hero}, (Theorem~\ref{thm:hero}) characterizes such sets of order 1. It then implies that such a (minimal) set of order 2 or more cannot contain a hero tournament.

\begin{figure}\label{fig:H123}
	\centering
	\begin{tikzpicture}
	[scale=1]
	\foreach \i / \j  in {0/1, 72/2, 144/3, 216/4, 288/5}
	{\draw[rotate=\i] (0,2)  node[circle,draw=black!80, line width=0.6mm, inner sep=0mm, minimum size=2.5mm] (x\j){};}
	
	\foreach \i / \j in {1/2,2/3,3/4,4/5,5/1}
	{\draw[->, line width=0.6mm, blue] (x\i) -- (x\j);}
	
	\foreach \i / \j in {1/3,2/4,3/5,4/1,5/2}
	{\draw[->, line width=0.6mm, blue] (x\i) -- (x\j);}

\begin{scope}[shift={(5,0)}]
	\foreach \i / \j  in {0/1, 72/2, 144/3, 216/4, 288/5}
{\draw[rotate=\i] (0,2)  node[circle,draw=black!80, line width=0.6mm, inner sep=0mm, minimum size=2.5mm] (x\j){};}

\foreach \i / \j in {1/2,2/3,3/4,4/5,1/5}
{\draw[->, line width=0.6mm, blue] (x\i) -- (x\j);}

\foreach \i / \j in {1/3,2/4,3/5,4/1,5/2}
{\draw[->, line width=0.6mm, blue] (x\i) -- (x\j);}

\end{scope}

\begin{scope}[shift={(10,0)}]

	\foreach \i / \j  in {0/1, 72/2, 144/3, 216/4, 288/5}
{\draw[rotate=\i] (0,2)  node[circle,draw=black!80, line width=0.6mm, inner sep=0mm, minimum size=2.5mm] (x\j){};}

\foreach \i / \j in {1/2,1/3, 1/4, 2/3,2/4,3/4, 1/5,3/5, 5/2, 5/4}
{\draw[->, line width=0.6mm, blue] (x\i) -- (x\j);}

\end{scope}

	\end{tikzpicture}
	\caption{$H_1$, $H_2$ and $H_3$}	
	
\end{figure}
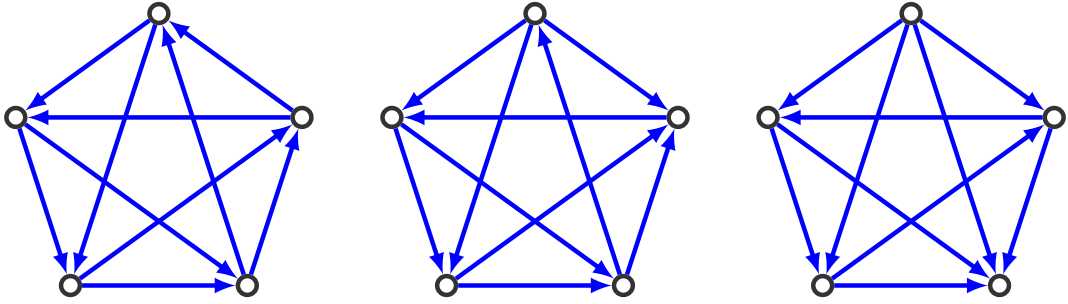

\medskip

An extension of Theorem~\ref{thm:hero} is given in \cite{HLNT} and will be extended in Theorem~\ref{thm:item2}.

\begin{theorem}[\cite{HLNT}] \label{thm:alpha}
For every integer $\alpha \geq 2$, the set $\{\ovlra K_2, \overline K_{\alpha}, H\}$
is heroic if and only if $H$ is a hero. 
\end{theorem}

\section{Digraphs that must be contained in all heroic sets}

One direction in the study of heroic sets is to find infinite families of digraphs whose dichromatic numbers increase. Given such a family, a heroic set $\mathcal{F}$ must contain one or more elements which removes the tail of the sequence from $\F(\mathcal{F})$. This would imply existence of certain types of digraphs in every heroic set. 
In this section, we give four such families. They can be seen as analogues of the two sequences used to prove the \textit{only if} part of Gy\'arf\'as-Sumner Conjecture (that is complete graphs and graphs with arbitrarily large girth and chromatic number). The situations is actually more complicated in the framework of directed graphs, and two more such sequences will be given in  Sections~\ref{sec:SecondItem} and~\ref{ss:ifC3}.

\begin{theorem}[Erd\"os \cite{Erd59}]\label{thm:HighGirth-HighChromatic}
Given positive integers $g$ and $k$ there exists a graph of girth at least $g$ and 
chromatic number $k$.
\end{theorem}

An analogue of this theorem for dichromatic number of oriented graphs is proved in \cite{girth}.

\begin{theorem}\cite{girth}\label{thm:HighGirth-HighDihromatic}
	Given positive integers $g$ and $k$ there exists an oriented graph whose underlying graph has girth at least $g$ and 
	whose dichromatic number is at least $k$.
\end{theorem}

Let us consider the four following families of digraphs:
 \begin{itemize}
  \item[1.] The family $\ovlra{K_{n}}$ of complete symmetric digraphs (because $\dic(\overleftrightarrow{K_n})=n$).
  \item[2.] The family $\ovlra{G_{{g,k}}}$, where for any positive integers $g$ and $k$, $G_{g,k}$ is a graph of girth at least $g$ and chromatic number at least $k$ (whose existence is guaranteed by Theorem~\ref{thm:HighGirth-HighChromatic}).   
  \item[3.] The family $H_{{g,k}}$, where for any positive integers $g$ and $k$, $H_{g,k}$ is an oriented graph whose underlying graph is of girth at least $g$ and whose dichromatic number is at least $k$ (whose existence is guaranteed by Theorem~\ref{thm:HighGirth-HighDihromatic}).
  
  \item[4.] The family of tournaments 
  \end{itemize}

Thus for the dichromatic of a class $\F(\mathcal{F} )$ of digraphs to be finite, 
$\mathcal{F}$ must contain at least four \textit{types} of elements:	

\begin{itemize}
 \item A complete symmetric digraph $\overleftrightarrow{K_k}$ for some integer $k$, because of the first family.
 \item A symmetric forest $ \ovlra{F_1}$ because of the second one (no other digraph can be a common induced subdigraph of the whole family because of the growing girth).
  \item An oriented forest $\ova{F_2}$ because of the third one (same argument).
  \item A tournament $T$ because of the fourth one (note that as explained in the previous subsection, if a heroic set contain a single tournament, then it must be a hero).
\end{itemize}



A digraph might be of several types. In particular, the digraph made of a single vertex is of all four types. We describe below all other digraphs that are of several types (actually each of them is of exactly two distinct types):

\begin{itemize}
 \item $\ovlra K_2$ is a complete symmetric digraph and a symmetric forest,
 \item for every integer $\alpha \geq 2$, $\overline{K}_{\alpha}$ is a symmetric forest, and an oriented forest,
 \item $\ova{K_2}$ is an oriented forest and a hero. 
\end{itemize}

As a direct consequence, we have a characterization of heroic set of size two. Note that $\F(\ovlra K_2, \ova{K_2})$ is the class of graphs with no arc, and is thus $1$-colorable. 

\begin{theorem}
The set $\{\ovlra K_2, \ova K_2\}$ is the unique heroic set of size two. 
\end{theorem}

The main study of this paper concerned heroic sets  of size three. If $\mathcal F$ is such a set, it results from above that $\mathcal F$ must contain at least one of $\ovlra K_2$, $\overline K_{\alpha}$, $\ova K_2$. 
If $\ova K_2$ belongs to $\mathcal F$, then we are  in the world of graphs (rather than digraphs) and thus we are in the framework of the well known Gy\'arf\'as-Sumner conjecture. Hence, we study the two other cases. More precisely, we ask:

\begin{enumerate}
\item[(\Pb1)]\label{P1} For  which hero $H$ and integers $k \geq 2$, $\alpha \geq 2$, is  $\{ \ovlra K_k,  \overline{K}_{\alpha}, H \}$   heroic?
\item[(\Pb2)]\label{P2} For which hero $H$ and oriented forest $F$, is $\{\ovlra K_2,H, {F}\}$ heroic?
\end{enumerate}

Observe that (\Pb2) is about oriented graphs and it is, in our opinion, of particular interest.

\section{Results and conjectures}

We completely settle (\Pb1):
\begin{theorem}\label{thm:item2}
Let $H$ be a hero and $k, \alpha \ge 2$ be integers. The set $\{\ovlra K_k, \overline K_{\alpha}, H\}$ is heroic if and only if $H$ is a transitive tournament or $k=2$
\end{theorem}
  
For (\Pb2), we venture to propose the following conjecture. A \textit{star} is a tree with at most one vertex of degree more than $1$. An \textit{oriented star} is an orientation of a star. By\textit{ disjoint union of oriented stars} we mean an oriented forest whose connected components are oriented stars. If $D_1$ and $D_2$ are two digraphs, we denote by $D_1+D_2$ the disjoint union of $D_1$ and $D_2$. 

\begin{conjecture}\label{conj:oriented}
Let $H$ be a hero and let $F$ be an oriented forest. The set $\{\ovlra K_2, H, F\}$ is heroic if and only if: either
\begin{itemize}
\item $F$ is the disjoint union of oriented stars,\\
 or
\item $H$ is a transitive tournament.
\end{itemize}
\end{conjecture}

We prove the \textit{only if} part of this conjecture in Section~\ref{ss:onlyif}. Chudnovsky, Scott and Seymour~\cite{CSS19} proved that, when both conditions hold, that is $H$ is a transitive tournament and $F$ is a disjoint union of oriented stars, the chromatic number of digraphs in  $\F\{\ovlra K_2, H, F\}$ is bounded, As the chromatic number of a digraph is an upper for its dichromatic number, this indeed implies that $\{\ovlra K_2, H, F\}$ is a heroic for such choices of $H$ and $F$.

In Sections~\ref{ss:ifC3} and \ref{ss:ifTT} we prove the following positive results. Before stating them, we need to extend our notation $\F(\mc F)$ by allowing (non-oriented) graphs in $\mc F$. If $\mc F$ is such a set, we define $\F(\mc F)$ to be the set of digraphs that does not contain as an induced subdigraph : any digraph of $\mc F$, and any orientation of a non-oriented graph of $\mc F$. For example $\F(K_3, \ovlra{K_2})$ is the class of triangle-free oriented graphs.

\begin{theorem}\label{thm:bounds}
The following sets are heroic:
\begin{itemize}
\item  $\{\ovlra K_2, \C{3}, \ova K_2 + K_1\}$,
\item  $\{\ovlra K_2, K_4, P^+(3)\}$, where $P^+(3)$ is the directed path of length $3$. 
\end{itemize}
\end{theorem}

The second claim of the  theorem is equivalent to stating that the set $\F$ of digraphs consisting of $\ovlra K_2$, $P^+(3)\}$ and all possible orientations of $K_4$ is a heroic set. As there are four non-isomorphic orientations on $K_4$, this is a set of order 6 and, at first look, it may seems that this does not fit into our direction of research which is about characterizing heroic sets of order 3. 
However, using the fact that every tournament on at least $2^{k}$ vertices contains $TT_k$ as an induced subdigraph, one easily observes that  given an oriented forest $F$ the a special case of our conjecture, that ``$\{\ovlra K_2, TT_k, F\}$ is a heroic set for every integer $k$", is equivalent to the following conjecture.

\begin{conjecture}\label{conj:Kk+F}
	Given an oriented forest $F$ and for every integer $l$,  $\{\ovlra K_2, K_k, F\}$ is heroic.
\end{conjecture}

\section{Proof of Theorem~\ref{thm:item2}}\label{sec:SecondItem}

As mentioned earlier, each tournament on at least $2^{k}$ vertices contains $TT_k$. Also, by classical Ramsey theory, for every integers $k_1,k_2,k_3$ there exists an integer $R(k_1,k_2,k_3)$ such that every $3$-coloring with colors $c_1, \, c_2,\, c_3$ of the edges of the complete graph on at least $R(k_1,k_2,k_3)$ vertices contains a monochromatic complete graph on $k_i$ vertices with color $c_i$ for some $i \in \{1,2,3\}$. 
Considering that a digraph is a $3$-coloring of the edges of a complete graph (colors being non-edge, induced arcs and digons), we thus have:
\begin{proposition}
Let $k$, $\alpha$, $t$ be integers. For every $D \in \F(\overleftrightarrow{K_k}, \overline{K}_{\alpha}, TT_t)$, $|V(D)| < R(k,\alpha,2^{t-1})$. In particular $\{\overleftrightarrow{K_k}, \overline{K}_{\alpha}, TT_t\}$ is a heroic set.
\end{proposition}

We now give a construction of digraphs with arbitrarily large dichromatic number proving that the set $\{\ovlra K_3, \overline K_2, H\}$ is not heroic as soon as $H$ is not a transitive tournament. 
Consider a symmetric graph $\ovlra G$ where $G$ is a graph of arbitrarily large girth and chromatic number. We fix an arbitrary enumeration $v_1, \dots, v_n$ of the vertices of $\ovlra G$ and create a semi-complete digraph\footnote{A \emph{semi-complete}  is a digraph with an arc or a digon between every pair of vertices} $D$ as follows: if $v_iv_j$, $i<j$, is a non-edge of $\ovlra G$, then $v_iv_j$ is an arc of $D$. Such a construction have arbitrarily large dichromatic number,  and belongs to $\F(\ovlra K_3, \overline K_2, \ova C_3)$. Moreover, Theorem~\ref{thm:hero} easily implies that a hero is either a transitive tournament, or contains a $\ova C_3$, so the set $\{\ovlra K_3, \overline{K_2}, H\}$ is not heroic as soon as $H$ is not a transitive tournament.  

Finally, since for all heroes $H$  the set $\{\ovlra K_2, \overline K_{\alpha}, H\}$ is a heroic (Theorem~\ref{thm:alpha}), Theorem~\ref{thm:item2} follows.


\section{Results about Conjecture~\ref{conj:oriented}}
In what follows, first of all we prove the only if part of Conjecture~\ref{conj:oriented}. Then, toward support for the main direction of this conjecture, we first recall that the case where both conditions are satisfied (i.e., $H$ is a  transitive tournament and $F$ is a union of oriented stars) is already settled in \cite{CSS19}. We then prove the conjecture for specific choices of $H$ and $F$ where only on of the two conditions are satisfied.

\subsection{The only if part}\label{ss:onlyif}
Let $D_1$ be the digraph on one vertex and let $D_{i+1}=C_{4}(D_{i})$.  From Theorem~\ref{coro:C_k(D)} it follows that $\dic(D_i)=i$. Furthermore, it is easily verified that  $D_i \in \F(\ovlra{K_2},\ova C_3, P_4)$ for every $i$.  

By Theorem \ref{thm:hero}, a hero with no $\ova C_3$ is a transitive tournament. Observe that an oriented forest with no induced $P_4$ in its underlying graph is a disjoint unions of oriented stars. This implies that if $H$ is a hero which is not a transitive tournament, and $F$ is a forest which is not a union of oriented stars, then $\F(\ovlra K_2, H)$ contains $\F(\ovlra{K_2},\ova C_3, P_4)$ whose chromatic number is not bounded. This proves the necessary part of Conjecture~\ref{conj:oriented}.

\subsection{The if part : $H=C_{3}$ and $F$ has at most $3$ vertices}\label{ss:ifC3}

Since it is proved that the set is heroic when both conditions hold, we can assume now only one of the two conditions holds. In this subsection we are interested in the case when $H$ is not transitive, which is equivalent to say that it contains a $C_{3}$, so the first step is then $H=C_{3}$ itself. 

If $F$ has only two vertices the question is trivial and on three vertices the possibilities for $F$ are :
\begin{itemize}
\item $\overbar K_{3}$, 
\item $P^{+}(2)$ (the directed path of length $2$)
\item  $\ova K_2 + K_1$, 
\item $S^{+}_{2}$ (the oriented star with two outgoing arcs from the center)
\item $S^{-}_{2}$ (the oriented star with two ingoing arcs to the center)
\end{itemize}
By Theorem \ref{thm:alpha}, $F= \overbar K_{3}$ gives a heroic set. The second item is simple : the digraphs in $\F(\ovlra K_2, \ova C_3, P^{+}(2))$ have dichromatic number $1$, as one can prove easily by induction that they consists of disjoint unions of transitive tournaments. 

\medskip

Let us know prove the case $F=\ova K_2 + K_1$ (the second item of Theorem \ref{thm:bounds}).

\begin{theorem}
 $\dic(\F(\ovlra K_2, \ova C_3, \ova K_2 + K_1))=2$. 
\end{theorem}

\begin{proof}
Let $D \in \F(\ovlra K_2, \ova C_3,  \ova K_2 + K_1)$. 
Observe that no induced $\ova K_2 + K_1)$ in an oriented graph implies that the set of non-neighbors of a vertex $x$ form an independents set. Thus non-adjacency is an equivalence relation on the set of vertices. In other words an oriented graph with no induced $\ova K_2 + K_1)$, and, in particular, our digraph $D$  is an orientation of a complete multipartite graph.
Furthermore, since $D$ has neither $\ovlra K_2$ nor $\ova C_3$ as induced subdigraph, all its induced directed cycles have length $4$ and all its $\ova C_4$'s are induced.  Hence, an acyclic coloring of $D$ is the same as a $\ova C_4$-free coloring, that is a coloring of the vertices such that no $\ova C_4$ is monochromatic.

Hence, we can look for a $\ova C_4$-free coloring with two colors. The proof is based on the following claim that says that all $\ova C_4$ containing a fixed vertex are included in the union of two parts of $D$.   

{\bf Claim:}
\emph{Let $a_1$ be a vertex of $D$. 
There exists $1 \le i \neq j \le n$ such that all $\ova C_4$ containing $a_1$ are included in $X_i \cup X_j$. }

Let  $a_1 \in X_i$ and assume $a_1$ is contained in a $\ova C_4$: $a_1\ra b_1\ra a_2 \ra b_2 \ra a_1$. Since this $\ova C_4$ is induced, and as the underlying graph of $D$ is a complete multipartite graph, we may assume that  $a_2 \in X_i$ and $b_1,b_2$ belong to a same part, say $X_j$, $j \neq i$. 
Assume now that $a_1$ belongs to another $\ova C_4$, say $a_1 \ra c_1 \ra a_3 \ra c_2$. 
Similarly, $a_3 \in X_i$ and $c_1,c_2$ are in a same part. 
Assume for contradiction that $c_1,c_2 \in X_k$ with $k \neq j$. In particular $b_1$ and $b_2$ are adjacent with both $c_1$ and $c_2$. 

If $a_2=a_3$, then either  $c_1 \ra b_2$ and $\{c_1, b_2, a_1\}$ induces a $\ova C_3$, or $b_2 \ra c_1$ and $\{b_2, c_1, a_2\}$ induces a $\ova C_3$, a contradiction in both cases. 

So $a_2 \neq a_3$. 
We then have:
\begin{itemize}
\item  $b_2 \ra c_1$, as otherwise $\{b_2,a_1,c_1\}$ induces a $\ova C_3$.
\item  $a_2 \ra c_1$, as otherwise $\{a_2,b_2,c_1\}$ induces a $\ova C_3$.
\item  $c_2 \ra b_1$, as otherwise $\{c_2,a_1,b_1\}$ induces a $\ova C_3$.
\item  $a_3 \ra b_1$, as otherwise $\{a_3,c_2,b_1\}$ induces a $\ova C_3$.
\end{itemize}
Now, if $b_1 \ra c_1$, then $\{b_1,c_1,a_3\}$ induces a $\ova C_3$ and if $c_1 \ra b_1$, then $\{b_1,c_1,a_2\}$ is a $\ova C_3$, a contradiction in both cases. This complete the proof of the claim.
\medskip

We can now partition each $X_i$  into $n$ subparts as follow. 
For $i,j=1, \dots, n$, $i \neq j$, define $X_i^j$ as the set of vertices in $X_i$ that are involved in $\ova C_4$ only with some vertices of $X_i \cup X_j$ and let $X_i^i$ the remaining set of vertices $X_i$ (those not in any $\ova C_4$). 

Build an auxiliary graph with vertices $x_i^j$ (representing the set $X_i^j$),   and put an edge between $x_i^j$ and $x_j^i$ for $1 \le i\neq j \le n$. It follows from the claim that each $X_i^j$ is adjacent to at most one vertex, that $X_j^i$ when $i\neq j$. Thus this graphs is disjoint union of $K_1$'s and $K_2$'s and, therefore, can be properly 2-colored. Take a $2$ coloring of it. Then, giving to all vertices of  $X_i^j$ the color of $x_i^j$, we obtain a  $\ova C_4$-free coloring of $D$ with $2$ colors. 
\end{proof}

To conclude the discussion started at the beginning of this section, we note that the next case where $H=\ova{C_3}$ and $F=S^{+}_{2}$ is already unsettled and we do not know if $\{\ovlra{K_2}, \ova{C_3},S^{+}_{2}\}$ is a heroic set, however we conjecture that:
 
\begin{conjecture}
$\dic(\F(\ovlra{K_2}, \ova{C_3},S^{+}_{2} ))=2$
\end{conjecture}

To support this, we point out that $\dic(\F(\ovlra{K_2},\ova{C_3}, S^{+}_{2},S^{-}_{2}))=2$. Indeed, in \cite{BG01} it is proven that the strongly connected element of $\F(\ovlra{K_2}, \ova{C_3}, S^{+}_{2},S^{-}_{2})$ is the class of so called {\em round digraphs}: a digraph is round if its vertices can be ordered cyclically $(v_1, v_2, \ldots v_n)$ such that whenever $v_iv_j$ is an arc, then for any $i<k<j$, $v_{i}v_{k}$ and $v_{k}v_{j}$ are both arcs (indices are taken modulo $n$). It is easy to see that round digraphs have dichromatic number $2$: consider the longest arc on the cyclic order, assume w.l.o.g. that it is $v_{1}v_{k}$, and then observe that $(\{v_{i},\, i\leq k\},\{v_{i},\, i> k\}$ is a partition into two acyclic digraphs.

\subsection{The if part : $H=TT_{k}$ and $F$ is an orientation of $P_{3}$}\label{ss:ifTT}
In this subsection we consider the case where $H$ {\em is} a transitive tournament and $F$ is an oriented forest {\em but not} a disjoint union of oriented stars. The smallest non trivial case is thus when $F$ is an orientation of a path on 4 vertices. 

Given the path $P$ on vertices  $v_1,v_2, \ldots, v_n$ where $v_{i}v_{i+1}$ are the edges, an orientation of $P$ can be coded by starting with a sign ($+$ or $-$) which decides the orientation of the first edge ($v_1v_2$) followed by a sequence of numbers, first of which tells the number of consecutive arcs in the same direction starting at $v_1$, then number of consecutive arcs at the opposite direction and so on. For example $P^{+}(3,4)$ is an orientation of a path of length 7 (8 vertices) where first 3 arcs are directed away from $v_1$ and last four are directed toward $v_1$. Using this terminology, the four orientations of the path on 4 vertices are represented on Figure~\ref{fig:orP4}.

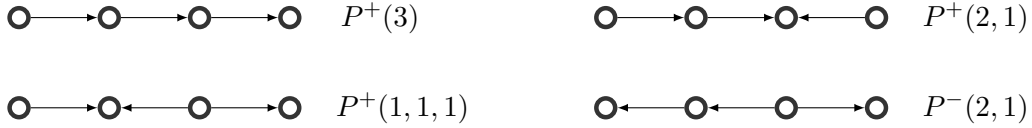
\begin{figure}[!hbtp]
\centering
 \begin{tikzpicture}[scale=0.6]
\node[draw,circle] (1) at (0,0) {};
\node[draw,circle] (2) at (2,0) {};
\node[draw,circle] (3) at (4,0) {};
\node[draw,circle] (4) at (6,0) {};

\draw[->,>=latex] (1) -- (2);
\draw[->,>=latex] (2) -- (3);
\draw[->,>=latex] (3) -- (4);

\draw (8,0) node[draw=none] {$P^+(3)$};

\begin{scope}[xshift=13cm]
\node[draw,circle] (1) at (0,0) {};
\node[draw,circle] (2) at (2,0) {};
\node[draw,circle] (3) at (4,0) {};
\node[draw,circle] (4) at (6,0) {};

\draw[->,>=latex] (1) -- (2);
\draw[->,>=latex] (2) -- (3);
\draw[->,>=latex] (4) -- (3);

\draw (8.2,0) node[draw=none] {$P^+(2,1)$};
\end{scope}

\begin{scope}[yshift=-2cm, xshift=13cm]
\node[draw,circle] (1) at (0,0) {};
\node[draw,circle] (2) at (2,0) {};
\node[draw,circle] (3) at (4,0) {};
\node[draw,circle] (4) at (6,0) {};

\draw[->,>=latex] (2) -- (1);
\draw[->,>=latex] (3) -- (2);
\draw[->,>=latex] (3) -- (4);

\draw (8.2,0) node[draw=none] {$P^-(2,1)$};
\end{scope}

\begin{scope}[yshift=-2cm]
\node[draw,circle] (1) at (0,0) {};
\node[draw,circle] (2) at (2,0) {};
\node[draw,circle] (3) at (4,0) {};
\node[draw,circle] (4) at (6,0) {};

\draw[->,>=latex] (1) -- (2);
\draw[->,>=latex] (3) -- (2);
\draw[->,>=latex] (3) -- (4);

\draw (8.5,0) node[draw=none] {$P^+(1,1,1)$};
\end{scope}

\end{tikzpicture}
  \caption{The four orientations of $P_4$}
    \label{fig:orP4}
\end{figure}

The chromatic number of the classes of oriented graphs where an orientation of $P_4$ is forbidden has been already studied. On the positive side, Chudnovsky et al \cite{CSS19} proved that $\F (\ovlra{K_2}, TT_k, P^+(2,1))$ has bounded chromatic number for every $k$, which implies that $\{\ovlra{K_2}, TT_k, P^+(2,1)\}$ is heroic (which by reversal of the arcs trivially implies the same for $\{\ovlra{K_2}, TT_k, P^-(2,1)\})$.

On the negative side, it is proved in respectively~\cite{ErHa76} and~\cite{KiTr92} that $\F(\ovlra K_2, K_3, P^+(1,1,1))$ and $\F(\ovlra K_2, K_3, P^+(3))$ have unbounded chromatic number (recall that forbidding $K_{3}$ here means we forbid both $TT_{3}$ and $\C{3}$, which is even stronger). 

We will prove that $\F(\ovlra K_2, K_3, P^+(3))$ has bounded dichromatic number. The rest of the subsection is dedicated to the proofs of $\dic(\F(\ovlra{K_2}, K_3, P^+(3)))=2$ and $\dic(\F(\ovlra{K_2}, K_4, P^+(3)))\leq 8$, but before we do so we need to introduce the notion of nice sets that will be the key tool of our proofs. 


\begin{definition}
Let $D$ be a digraph. A nonempty set of vertices $S$ of $D$  is said to be \emph{nice} if  each vertex in $S$ either  has no  out-neighbor  in $V(D) \sm S$ or  has no in-neighbor in $V(D) \sm S$. 
The  set of vertices in $S$ with no out-neighbor in $V(D) \sm S$ is called the \emph{in-part of $S$}, and the  set of vertices in $S$ with no  in-neighbor in $V(D) \sm S$ is the \emph{out-part of $S$}. 
\end{definition}

The next lemma gives a sufficient condition for a class of digraphs to have bounded dichromatic number. 

\begin{lemma}\label{lem:niceSet}
Let $\mc C$ be a hereditary class of digraphs. Assume that there exists two integers $c_1$ and $c_2$ such that every digraph in $\mc C$ contains a nice set $S$ such that the  in-part of $S$ has dichromatic number at most $c_1$ and its out-part has dichromatic number at most $c_2$.   Then $\dic(\mc C) \le c_1+c_2$. In particular, if there exists $c$ such that every digraph in $\mc C$ admits a nice set with dichromatic number at most $c$, then $\dic(\mc C) \le 2c$. 
\end{lemma}

\begin{proof}
Let $\mc C$ be a  class of digraph as in the statement. Let $D \in \mc C$ be a minimal counter example, that is: $\dic(D)=c_1+c_2+1$ and for every proper  subdigraph $H$ of $D$, $\dic(H)\le c_1+c_2$. 
By hypothesis, $D$ admits a nice set $S$, with in-part $S_1$  and out-part $S_2$ such that $\dic(S_1)\leq c_1$ and $\dic(S_2)\leq c_2$. 

The key observation is that any directed cycle that intersects $S$ and $V(D) \sm S$ must intersect  both $S_1$ and $S_2$.
Hence, by the minimality of $D$, we can dicolor the subdigraph of $D$ induced by $V(D) \sm S$ with $c_1+c_2$ colors. We can then extend this dicoloring to $D$ by using colors $1, \dots, c_1$ for $S_1$ and $c_1+1, \dots, c_1 + c_2$ for $S_2$. 
\end{proof}

Let $D$ be an oriented graph and let $x,y$ be two vertices of $D$. 
The \emph{distance} between   $x$ and $y$ is the distance between $x$ and $y$ in the underlying graph of $D$. 
The \emph{out-distance} from   $x$ to $y$ is the length of a shortest directed path from $x$ to $y$. 
The \emph{in-distance} from  $x$ to $y$ is the length of a shortest directed path from $y$ to $x$. 

\begin{theorem}\label{thm:K3P3}
 $\dic(\F(\ovlra{K_2}, K_3, P^+(3))=2$.  
\end{theorem}

\begin{proof}
Let $D \in \F(\ovlra{K_2}, K_3, P^+(3))$. Assume $D$ is strongly connected (otherwise just take the strong connected component with largest dichromatic number). 
Let $x \in V(D)$.  
For $i \geq 0$, set $L_i$ to be the set of vertices at out-distance $i$ from $x$. Since $D$ is strongly connected, the collection of $L_i$'s is a partition of $V(D)$. 
We are going to prove that each layer induces a stable set. 
Let $k$ be the maximum integer such that each $L_i$ is a stable set for  $i=1, \dots, k$. 
Since $D$ is $K_3$-free, $L_1$ is a stable set, so $k \ge 1$. 
If $L_{k+1}$ is empty, we are done. So assume $L_{k+1}$ is not empty, and by maximality of $k$, $L_{k+1}$ contains an arc $ab$. 
There exists $a_1 \in L_{k}$ and $a_2 \in L_{k-1}$ such that $a_2 \ra a_1 \ra a$. 
Since $a_2\ra a_1 \ra a \ra b$ cannot be induced and $a_1$ and $b$ are non-adjacent (because $D$ is triangle-free),  $b \ra a_2$. 
There exists $b_1 \in L_{k}$ and $b_2 \in L_{k-1}$ such that $b_2 \ra b_1 \ra b$.  Since $D$ is $K_3$-free, $b_1 \neq a_1$, $b_2 \neq a_2$ and $b_1$ is not adjacent with $a_2$. 
Moreover, since $L_{k-1}$ is a stable set, $a_2$ is not adjacent with  $b_2$. Hence $b_2b_1ba_2$ is an induced $P^{+}(3)$, a contradiction. 

Now,  color every vertex at odd out-distance from $x$ with color 1, every vertex at even out-distance from $x$ with color $2$ and $x$ with color $2$. It is easy to check that this gives a proper dicoloring.
\end{proof}

Our next goal is to prove that $\{\ovlra{K_2}, K_4, P^+(3)\}$ is heroic. In order to do so, we first prove that two larger sets are heroic, namely  
$\{\ovlra{K_2}, K_4, P^+(3), \ova{C_3}\}$ (see Lemma~\ref{lem:P3C3}) and  
$\{\ovlra{K_2}, K_4, P^+(3), R\}$, where $R$ is the graph depicted in Figure~\ref{figR} (Lemma~\ref{lem:Rfree}). 

\begin{lemma}\label{lem:P3C3}
$\dic(\F(\ovlra{K_2}, K_4, P^+(3), \ova{C_3})) \le 8$ 
\end{lemma}
\begin{proof} 
We are going to prove that every graph in $\F(\ovlra{K_2}, K_4, P^+(3), \ova C_3))$ contains a nice set with in-part and out--part of dichromatic number at most $4$. This would imply the claim of this Lemma using Lemma~\ref{lem:niceSet}. 

Let $D \in \F(\ovlra{K_2}, K_4, P^+(3), \ova C_3)$. 
Let $x \in V(D)$. Let $X_1$ (resp.~$X_2$) be the vertices at distance $1$ (resp. at distance $2$) from $x$. Let $X_1^+=N^+(x)$ and let $X_1^-=N^-(x)$ (so $X_1=X_1^+ \cup X_1^-$). Let $X_2^+=X_2 \cap N^+(X_1^+)$ and $X_2^-=X_2 \cap N^-(X_1^-)$. Observe that $X_2^+ \cup X_2^-$ does not need to be equal to $X_2$. 

Let us prove that $S=\{x \} \cup X_1^+ \cup X^-_1 \cup X^+_2 \cup X_2^-$ is a nice set with in-part $\{x\} \cup X_1^+ \cup X^+_2$ and out-part $X_1^- \cup X_2^-$. By definition of $S$, it is clear that $x$ has no neighbor in $V(D) \sm S$, that vertices in $X_1^+$ have no out-neighbor in $V(D) \sm S$ and that vertices in $X_1^-$ have no in-neighbor in $V(D) \sm S$. 
Let $x_2 \in X_2^+$ and let us prove that $x_2$ have no out-neighbor in $V(D) \sm S$. Assume for contradiction that there exists $x_3 \in V(D) \sm S$ such that $x_2 \ra x_3$. By definition of $X_2^+$ there exists a vertex $x_1\in X_1^+$ such that $x \ra x_1 \ra x_2$. 
If $x_3 \ra x_1$, then $x_1x_2x_3$ is a $\ova C_3$ and if $x_1 \ra x_3$, then $x_2 \in X^+_2$, a contradiction in both cases, so $x_1$ and $x_3$ are non-adjacent and thus $x \ra x_1 \ra x_2\ra x_3$ is induced, a contradiction. This proves that vertices in $X^+_2$ have no out-neighbor in $V(D) \sm S$. Similarly (because $P^+(3)$ is invariant under reversing all edges), $X_2^-$ have no in-neighbor in $V(D) \sm S$. This proves that $S$ is a nice set with in and out-part as announced. 

We now prove that $\{x\} \cup X_1^+ \cup X^+_2$ and $X_1^- \cup X_2^-$ are $4$-dicolorable. 
Since $D$ is $K_4$-free, $X_1$ is triangle-free and is thus $2$-dicolorable by Theorem~\ref{thm:K3P3}. 
Assume that $X_2^+$ has a $TT_3$, say $a \ra b \ra c \leftarrow a$. By definition of $X_2^+$, there is a vertex $x_1 \in X_1^+$ such that $x_1 \ra a$. Since $x \ra x_1 \ra a \ra b$ cannot be induced, $x_1$ and $b$ must be adjacent and since $D$ has no $\ova C_3$, $x_1 \ra b$. Since $x \ra x_1 \ra b \ra c$ cannot be induced, $x_1$ is also adjacent with $c$, and thus $\{a,b,c,x_1\}$ induces a $K_4$, a contradiction. Hence $X_2^+$ is  triangle-free and thus $2$-dicolorable. Similarly, $X_2^-$ is $2$-dicolorable. 
We may now use two colors on $ X_1^+$ and distinct set of two colors on  $X_2^+$, then use any of the four colors to color $x$. As $x$ is in no direct 4-cycle induced by $\{x\} \cup X_1^+ \cup X^+_2$, this ia 4-dicoloring of this induced subgraphs. That $X_1^- \cup X_2^-$ is $4$-dicolorable is proved analogously.
\end{proof}

\begin{figure}
\begin{center}

\begin{tikzpicture}[scale=0.9]

\vertex[minimum size=11pt](a) at  (0,0) {$a$}; 
\vertex[minimum size=11pt](b) at  (2,0) {$b$}; 
\vertex[minimum size=11pt](c) at  (4,0) {$c$}; 
\vertex[minimum size=11pt](d) at  (0,2) {$d$}; 
\vertex[minimum size=11pt](e) at  (4,2) {$e$}; 
\vertex[minimum size=11pt](f) at  (0,4) {$f$}; 
\vertex[minimum size=11pt](g) at  (4,4) {$g$};

\draw [->,>=latex] (a) to  (b);  
\draw [->,>=latex] (b) to  (c);  
\draw [->,>=latex] (c) to [bend left=30]  (a);  

\draw [->,>=latex] (b) to  (d);  
\draw [->,>=latex] (d) to  (a);  
\draw [->,>=latex] (c) to  (e);  
\draw [->,>=latex] (e) to  (b);  

\draw [->,>=latex] (f) to  (d);  
\draw [->,>=latex] (d) to  (g);  
\draw [->,>=latex] (g) to  (f);  
\draw [->,>=latex] (f) to  (e);  
\draw [->,>=latex] (e) to  (g);  

\end{tikzpicture}
\caption{The oriented graph $R$. Observe that $R$ is invariant under reversing all its arcs}
\label{figR}
\end{center}
\end{figure}
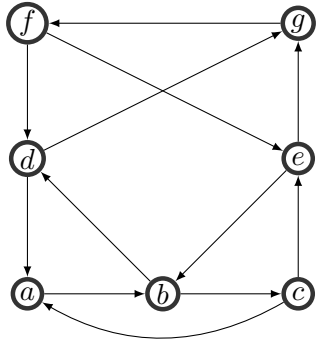

We now need a second technical lemma. 
Let us first define a particular class of oriented graphs named $\mc F$. An oriented graph $F$ belongs to $\mc F$ if there exists $D_F \in \F(\ovlra{K_2}, K_4, P^+(3))$ such that $D_F$ is made of a copy $F'$ of $F$, a stable set of vertices $L$ (disjoint from $V(F')$) such that every vertex in $F'$ has at least one neighbor in $L$, and two more vertices $u$ and $v$ (outside $V(L) \cup V(F')$) such that $u \ra v$, and for every vertex $x$ in $L$ we have $v \ra x \ra u$, and there is no arc between $\{u,v\}$ and $V(F')$.  Observe that $\mc F$ is  hereditary  and is a subclass of $\F(\ovlra{K_2}, K_4, P^+(3))$. 

\begin{lemma}\label{lem:Rfree}
Graphs in $\mc F$ are $R$-free, where $R$ is the graph depicted in Figure~\ref{figR}.
\end{lemma}
\begin{proof}
Let $F \in \mc F$ and let $D_F$ be the graph described as in the definition of $\mc F$ above. Assume for contradiction that $F$ contains an induced copy of $R$ (with same name for vertices as in Figure~\ref{figR}). By the definition of $D_F$, $b$ has a neighbor $x$ in $L$. Recall that there are two vertices $u$ and $v$ in $D_F$ such that $u \ra v \ra x \ra u$ and such that there is no arc between $\{u,v\}$ and $V(R)$. By reversing all edges of $D_F$, if necessary, we may assume that $x \ra b$ (this is legitimate since,  $R$, $\ova C_3$ and $P^+(3)$ are invariant under reversing all edges).  In the upcoming case analysis, we will be using the fact that there is no arc between $\{u,v\}$ and $V(R)$ without recalling this fact.   

Since $v \ra x \ra b \ra d$ cannot be induced,  $x$ and $d$ are adjacent. Since $D_F$ is $K_4$-free, $x$ and $a$ are non-adjacent. If $x \ra d$, then $v \ra x \ra d \ra a$ is induced, a contradiction. Hence $d \ra x$ holds.

Since $f \ra d \ra x \ra u$ cannot be induced,  $x$ and $f$ must be adjacent. Since $D_F$ is $K_4$-free, $x$ and $g$ are non-adjacent. If $f \ra x$, then $g \ra f \ra x \ra u$ is induced, a contradiction. Hence $x \ra f$ holds. 

Finally, since $v \ra x \ra f \ra e$ cannot be induced,  $x$ and $e$ are adjacent, and since $v \ra x \ra b \ra c$ cannot be induced,  $x$ and $c$ are adjacent. Hence $\{x,b,c,e\}$ induces a $K_4$, a contradiction. 
\end{proof}

\begin{lemma}\label{lem:P3R}
$\dic(\F(K_4, P^+(3), R))\le 66$
\end{lemma}
\begin{proof}
We are going to prove that every graph in $\F(K_4, P^+(3), R)$ contains a nice set   of dichromatic number at most $33$, which, by Lemma~\ref{lem:niceSet}, implies our claim. 

Let $D \in \F(K_4, P^+(3), R)$. If $D$ is $\ova C_3$-free, we are done by Theorem~\ref{thm:K3P3}.
So we may assume that $D$ contains $C= u \ra v \ra w \ra u$. 
Let $X^{uv}_1$ (resp. $X^{vw}_1$, resp. $X^{wu}_1$) be the set of vertices $x \in N(C)$  such that $v \ra x \ra u$ 
(resp. such that $w \ra x \ra v$, resp. $u \ra x \ra w$). Let $X_1=N(C) \sm (X^{uv}_1 \cup X^{vw}_1 \cup X^{wu}_1)$. 
Let $X_2^{uv}$ (resp. $X_2^{vw}$, resp. $X_2^{wv}$) be the set of vertices in $V(D) \sm (V(C) \cup N(C))$  having at least one neighbor in $X_1^{uv}$ (resp. in $X_1^{vw}$, resp. in $X_1^{wv}$). 

Let us prove that the set $S=\{u,v,w\} \cup X^{uv}_1 \cup X^{vw}_1 \cup X^{wu}_1 \cup X_1 \cup X^{uv}_2 \cup X^{vw}_2 \cup X^{wu}_2$ is a nice set. We say that a vertex is \emph{nice} if it has no out-neighbor or no in-neighbor in $V(D) \sm S$. First observe that the neighborhood of a vertex in $\{u,v,w\} \cup X^{uv}_1 \cup X^{vw}_1 \cup X^{wu}_1$ is included in $S$, and thus every  vertex in  $\{u,v,w\} \cup X^{uv}_1 \cup X^{vw}_1 \cup X^{wu}_1$ is nice.

Let $x_1 \in X_1$ and let us prove that $x_1$ is nice. We are in one of the three following situations:
\begin{itemize}
\item $x_1$ has only in-neighbors in $C$. Assume with out loss of generality that $u \ra x_1$. We claim that $x_1$ has no out-neighbor in $V(D) \sm S$. Assume for contradiction that there exists $x_2 \in V(D) \sm S$ such that $x_1 \ra x_2$. By construction of $S$, $x_2 \notin N(C)$. Since $w \ra u \ra x_1 \ra x_2$ cannot be induced, $w$ and $x_1$ must be adjacent and thus $w \ra x_1$. Since $v \ra w \ra x_1 \ra x_2$ cannot be induced, $v$ and $x_1$ must be adjacent, but then $\{u,v,w, x_1\}$ induces a $K_4$, a contradiction. This proves the announced claim. 
\item $x_1$ has only out-neighbors in $C$. In this case,  $x_1$ has no in-neighbor in $V(D) \sm S$, we skip the proof that is similar to the one of the previous case. 
\item $x_1$ has both an in-neighbor and an out-neighbor in $C$, and is not forming a $\ova C_3$ with arcs of $C$. We may assume, without loss of generality, that $u \ra x_1 \ra v$, and observe that in this case $w$ and $x$ are not adjacent because the underlying graph of $D$ is $K_4$-free.
We claim that $x_1$ has no out-neighbor in $V(D) \sm S$ (it has actually no neighbor at all in $V(D) \sm S$). Assume for contradiction that there exists $x_2$ in $V(D) \sm S$ such that $x_1\ra x_2$. By construction of $S$, $x_2 \notin V(C) \cup N(C)$, and thus $w \ra u \ra x_1 \ra x_2$ is induced, a contradiction. 
\end{itemize}
This proves that every vertex in $X_1$ is nice. 

Let us now prove that every vertex in $X_2^{uv}$ is nice. Let $x_2 \in X_2^{uv}$. By definition of $X_2^{uv}$, there is a vertex $x_1 \in X_1^{uv}$ such that $x_1$ and $x_2$ are adjacent. Observe that $x_1$ and $w$ are non-adjacent (because $D$ is $K_4$-free). Let $x_3$ be a neighbor of $x_2$ in $V(D) \sm S$. Observe that by the definition of $S$, $x_1$ and $x_3$ are non-adjacent. 
Hence, if $x_1 \ra x_2$, then $x_2 \ra x_3$ cannot hold, as otherwise $v \ra x_1 \ra x_2 \ra x_3$ is induced, and if $x_2 \ra x_1$, then $x_3 \ra x_2$ cannot hold, otherwise $x_3 \ra x_2 \ra x_1 \ra u$ is induced. 
This proves that $x_2$ is nice. The situation in $X_2^{vw}$ and $X_2^{wu}$ being exactly the same, every vertex in $X_2^{vw} \cup X_2^{wu}$ is also nice, and thus  $S$ is a nice set. 
\medskip

It now remains to prove that $\dic(S) \le 33$. Observe that $D$ being $K_4$-free, the neighborhood of a vertex of $D$ is $K_3$-free and thus $2$-dicolorable by Theorem~\ref{thm:K3P3}.  Hence $C \cup N(C)=\{u,v,w\} \cup X^{uv}_1 \cup X^{vw}_1 \cup X^{wu}_1 \cup X_1$ is $9$-dicolorable.  
It is thus enough to show that $X_2^{uv} \cup X_2^{vw} \cup X_2^{wu}$ is $24$-dicolorable, and by symmetry between $X_2^{uv}$,  $X_2^{vw}$  and $X_2^{wu}$,  it is enough to show that $X_2^{uv}$ is $8$-dicolorable.  
Finally, by Lemma~\ref{lem:P3C3}, it is enough to show that $X_2^{uv}$ is $\ova C_3$-free. 
Assume for contradiction that $X^{uv}_2$ contains a $\ova C_3$, say $a \ra b \ra c \ra a$. 
Let us prove a technical claim:
\medskip

\noindent{\bf Claim:} {\em
 Let $x_1 \in X_1^{uv}$ such that $x_1$ has a neighbor in $\{a,b,c\}$. Then $x_1$ has exactly two neighbors in $\{a,b,c\}$ and forms a $\ova C_3$ with these two neighbors.}
\medskip

\noindent{\em Proof of Claim:} 
Assume without loss of generality that $x_1$ and $a$ are adjacent. Since all the reasoning will be on $D[\{u,v,x_1,a,b,c\}]$ which is invariant under reversing all edges, and what we want to prove is also invariant under reversing all edges,  we can assume without loss of generality that $x_1 \ra a$. Since $v \ra x_1 \ra a \ra b$ cannot be induced, $x_1$ and $b$ are adjacent. Since $D$ is $K_4$-free, $x_1$ and $c$ are non-adjacent. If $x_1 \ra b$, then $v \ra x_1 \ra b \ra c$ is induced, a contradiction. Hence $b \ra x_1$. 
This completes the proof of Claim.
\medskip

By definition of $X^{uv}_2$, there exists $x_1 \in X_1^{uv}$ such that $x_1$ and $a$ are adjacent. By \emph{Claim}, we can assume without loss of generality that $b \ra x_1 \ra a$ and that $x_1$ and $c$ are non-adjacent. Hence, there exists a vertex $x_1'$ in $X_1^{uv} \sm \{x_1\}$ such that $x'_1$ and $c$ are adjacent. Observe that $x_1$ and $x'_1$ are non-adjacent, otherwise $D[\{u,v,x_1,x'_1\}]$ is a $K_4$, a contradiction.  By \emph{Claim}, we either have $c \ra x'_1 \ra b$ and $x_1'$ and $a$ are non-adjacent,  or $a\ra x'_1 \ra c$ and $x'_1$ and $b$ are non-adjacent. In both cases $D[\{u,v,x_1,x'_1,a,b,c\}]$ induces $R$, a contradiction. Hence, $X^{uv}_2$ is $\ova C_3$-free.
\end{proof}

\begin{theorem}\label{thm:K4P3}
$\dic(\F(K_4, P^+(3))) \le 414$.
\end{theorem}

\begin{proof}
By Lemma~\ref{lem:niceSet} it would be enough to prove that $\F(K_4, P^+(3))$ contains a nice set   of dichromatic number at most $207$.

Let $D \in  \F(K_4, P^+(3)))$. If $D$ is $\ova C_3$-free, we are done by Lemma~\ref{lem:P3C3}.
So we may assume that $D$ contains $C= u \ra v \ra w \ra u$. 

Define $S=\{u,v,w\} \cup X^{uv}_1 \cup X^{vw}_1 \cup X^{wu}_1 \cup X_1 \cup X^{uv}_2 \cup X^{vw}_2 \cup X^{wu}_2$ exactly as in the proof of Lemma~\ref{lem:P3C3}. 
As in the proof of Lemma~\ref{lem:P3R}, $S$ is a nice set and  $C \cup N(C)=\{u,v,w\} \cup X^{uv}_1 \cup X^{vw}_1 \cup X^{wu}_1 \cup X_1$ is $9$-dicolorable. 
It is thus enough to show that $X_2^{uv} \cup X_2^{vw} \cup X_2^{wu}$ is $198$-dicolorable, and by symmetry between $X_2^{uv}$,  $X_2^{vw}$  and $X_2^{wu}$,  it is enough to show that $X_2^{uv}$ is $66$-dicolorable.  
By the construction of $X_2^{uv}$, $D[X_2^{uv}]$ is in $\mc F$, hence by Lemma~\ref{lem:Rfree}, it is $R$-free and thus $66$-colorable by Lemma~\ref{lem:P3R}.\end{proof}

{\bf Acknowledgment} This project was financed by the ANR projects DISTANCIA (ANR-17-CE40-0015),  HOSIGRA (ANR-17-CE40-0022) and ALGORIDAM (ANR-19-CE48-0016).
We would also like to thanks Maria Abi Aad  and Mekkia Kouider  for participation in discussions at the early stages of the project.

\end{document}